\documentclass[final,leqno]{siamltex}
\usepackage{epsfig}
\usepackage{amsmath}
\usepackage{amssymb}
\usepackage{tikz}
\usepackage{graphicx}
\usepackage[notcite,notref]{showkeys}
\newtheorem{algorithm}{Weak Galerkin Algorithm}
\newcommand{\bq}{{\bf q}}
\newcommand{\bn}{{\bf n}}
\newcommand{\bx}{{\bf x}}

\def\T{{\mathcal T}}
\def\E{{\mathcal E}}
\def\Q{{\mathbb Q}}

\def\l{{\langle}}
\def\r{{\rangle}}

\def\bn{{\bf n}}
\def\bq{{\bf q}}

\newcommand{\pT}{{\partial T}}

\def\3bar{{|\hspace{-.02in}|\hspace{-.02in}|}}

\title{Stabilizer-free weak Galerkin finite element
   methods on polytopal meshes}

\author{Xiu Ye\thanks{Department of
Mathematics, University of Arkansas at Little Rock, Little Rock, AR
72204 (xxye@ualr.edu). This research was supported in part by
National Science Foundation Grant DMS-1620016.}
\and
Shangyou Zhang\thanks{Department of
Mathematical Sciences, University of Delaware, Newark, DE 19716 (szhang@udel.edu).}
}

\begin{document}

\maketitle

\begin{abstract}
A stabilizing/penalty term  is often used in finite element methods
    with discontinuous approximations  to enforce  connection of discontinuous functions
   across element boundaries.
Removing  stabilizers from discontinuous  Galerkin finite element methods will simplify
   formulations and reduce programming complexity significantly.
The goal of this paper is to introduce a stabilizer free weak Galerkin (WG) finite element method
    for second order elliptic equations on polytopal meshes.
This new WG method keeps a simple symmetric positive definite form and can work on polygonal/polyheral
     meshes. Optimal order error estimates are established for the corresponding
WG approximations in both a discrete $H^1$ norm and the  $L^2$ norm.
Numerical results are presented verifying the theorem.
\end{abstract}


\pagestyle{myheadings}

\section{Introduction}\label{Section:Introduction}

We consider Poisson equation with a homogeneous Dirichlet boundary condition in $d$ dimension
  as our model problem for the sake of clear presentation.
This stabilizer free weak Galerkin  method can also be used for
   other partial differential equations.
The Poisson problem  seeks an unknown function $u$ satisfying
\begin{eqnarray}
-\Delta u&=&f\quad \mbox{in}\;\Omega,\label{pde}\\
u&=&0\quad\mbox{on}\;\partial\Omega,\label{bc}
\end{eqnarray}
where $\Omega$ is a polytopal domain in $\mathbb{R}^d$.

The weak form of the problem (\ref{pde})-(\ref{bc}) is to find $u\in H^1_0(\Omega)$
such that
\begin{eqnarray}
(\nabla u,\nabla v)=(f,v)\quad \forall v\in
H_0^1(\Omega).\label{weakform}
\end{eqnarray}

The $H^1$ conforming finite element method for the problem (\ref{pde})-(\ref{bc}) keeps the same simple form as in (\ref{weakform}): find $u_h\in V_h\subset H^1_0(\Omega)$
such that
\begin{eqnarray}
(\nabla u_h,\nabla v)=(f,v)\quad \forall v\in V_h,\label{cfe}
\end{eqnarray}
where $V_h$ is a finite dimensional subspace of $H_0^1(\Omega)$.
The functions in $V_h$ are required to be continuous, which makes the
  classic  finite element formulation (\ref{cfe}) less flexible
   in element constructions and in mesh generations.
In contrast, finite element methods using discontinuous approximations
   have two advantages: 1. easy construction of high order elements
    and  avoiding constructing some special elements such as $C^1$ conforming elements;
  2. easy working on general meshes.
Therefore, discontinuous finite element methods are the most active research area
    in the context of finite element methods for the past two decades.
Discontinuous approximation was first used in finite element procedure as early as in
   1970s \cite{Babu73, DoDu76,ReHi73, Whee78}.
Local discontinuous Galerkin methods were introduced in \cite{cs1998}.
 Then a paper \cite{abcm} in 2002 provides a unified analysis of discontinuous Galerkin
     finite element methods for Poisson equation.
More discontinuous finite element methods  have been developed such as
   hybridizable discontinuous Galerkin  method \cite{cgl},
    mimetic finite differences method \cite{Lipnikov2011},
hybrid high-order   method \cite{de},
 weak Galerkin  method \cite{wy}  and references therein.

One obvious disadvantage of discontinuous finite element methods is their rather complex
   formulations which are often necessary to enforce weak continuity of discontinuous solutions
    across element boundaries.
Most of discontinuous finite element methods have one or more stabilizing terms
     to guarantee  stability and convergence of the methods.
Existing of stabilizing terms further complicates formulations.
Complexity of discontinuous finite element methods makes them difficult
   to be implemented and to be analyzed.
The purpose of this paper is to obtain a finite element
    formulation close to its original PDE weak form (\ref{weakform})
    for discontinuous polynomials.
We believe that finite element formulations for discontinuous approximations
     can be as simple as follows:
\begin{equation}\label{dfe}
(\nabla_w u_h,\nabla_w v)=(f,v),
\end{equation}
if  $\nabla_w$, an approximation of gradient,  is appropriately defined.
The formulation (\ref{dfe}) can be viewed as the counterpart of  (\ref{weakform}) for discontinuous approximations.
In fact such an ultra simple formulation (\ref{dfe}) has been achieved
     for one kind of WG method in \cite{wy},
    and for the conforming DG methods in \cite{cdg1,cdg2}.
 The lowest order WG method developed in \cite{wy} has been improved
    in \cite{liu} for convex polygonal meshes,
   in which non-polynomial functions are used for computing weak gradient.

In this paper, we develop a  WG finite element method that
    has an ultra simple formulation (\ref{dfe}) and can work on polytopal meshes for any polynomial degree $k\ge 1$.
The idea is to raise the degree of polynomials used to compute weak gradient $\nabla_w$.  Using higher degree polynomials in computation of weak gradient will not change
the size, neither the global sparsity of the stiffness matrix. On the other side, the simple formulation of the stabilizer free WG method (\ref{dfe}) will reduce programming complexity significantly.
Optimal order error estimates are established for the corresponding
   WG approximations in both a discrete $H^1$ norm and the  $L^2$ norm.
Numerical results are presented verifying the theorem.

\section{Weak Galerkin Finite Element Schemes}\label{Section:wg-fem}

Let ${\cal T}_h$ be a partition of the domain $\Omega$ consisting of
polygons in two dimension or polyhedra in three dimension satisfying
a set of conditions specified in \cite{wymix}. Denote by ${\cal E}_h$
the set of all edges or flat faces in ${\cal T}_h$, and let ${\cal
E}_h^0={\cal E}_h\backslash\partial\Omega$ be the set of all
interior edges or flat faces. For every element $T\in \T_h$, we
denote by $h_T$ its diameter and mesh size $h=\max_{T\in\T_h} h_T$
for ${\cal T}_h$.

We start by introducing weak function $v=\{v_0,v_b\}$ on element $T\in\T_h$ such that
$$
v=
\left\{
\begin{array}{l}
  \displaystyle
  v_0\quad {\rm in}\; T,
  \\ [0.08in]
  \displaystyle
  v_b\quad {\rm on}\;\partial T.
 \end{array}
\right.
$$
If $v$ is continuous on $\Omega$, then $v=\{v,v\}$.

For a given integer $k \ge 1$, let $V_h$ be the weak Galerkin finite
element space associated with $\T_h$ defined as follows
\begin{equation}\label{vhspace}
V_h=\{v=\{v_0,v_b\}:\; v_0|_T\in P_k(T),\ v_b|_e\in P_{k}(e),\ e\subset\pT,  T\in \T_h\}
\end{equation}
and its subspace $V_h^0$ is defined as
\begin{equation}\label{vh0space}
V^0_h=\{v: \ v\in V_h,\  v_b=0 \mbox{ on } \partial\Omega\}.
\end{equation}
We would like to emphasize that any function $v\in V_h$ has a single
value $v_b$ on each edge $e\in\E_h$.

For given $T\in\T_h$ and $v=\{v_0,v_b\}\in V_h+H^1(\Omega)$, a weak gradient $\nabla_w v \in [P_j(T)]^d$ ($j > k$) is defined as the unique polynomial satisfying
\begin{equation}\label{d-d}
  (\nabla_w v, \bq)_T = -(v_0, \nabla\cdot \bq)_T+ \langle v_b, \bq\cdot\bn\rangle_{\partial T}\qquad
   \forall \bq\in [P_j(T)]^d,
\end{equation}
where $j$ will be specified later.

Let $Q_0$ and $Q_b$ be the two element-wise defined $L^2$ projections onto $P_k(T)$ and $P_k(e)$ with $e\subset\partial T$ on $T$ respectively. Define $Q_hu=\{Q_0u,Q_bu\}\in V_h$. Let $\Q_h$ be the element-wise defined $L^2$ projection onto $[P_{j}(T)]^d$ on each element $T$.

For simplicity, we adopt the following notations,
\begin{eqnarray*}
(v,w)_{\T_h} &=&\sum_{T\in\T_h}(v,w)_T=\sum_{T\in\T_h}\int_T vw d\bx,\\
 \l v,w\r_{\partial\T_h}&=&\sum_{T\in\T_h} \l v,w\r_\pT=\sum_{T\in\T_h} \int_\pT vw ds.
\end{eqnarray*}

\begin{algorithm}
A numerical approximation for (\ref{pde})-(\ref{bc}) can be
obtained by seeking $u_h=\{u_0,u_b\}\in V_h^0$
satisfying  the following equation:
\begin{equation}\label{wg}
(\nabla_wu_h,\nabla_wv)_{\T_h}=(f,\; v_0) \quad\forall v=\{v_0,v_b\}\in V_h^0.
\end{equation}
\end{algorithm}

\begin{lemma}
Let $\phi\in H^1(\Omega)$, then on any $T\in\T_h$,
\begin{equation}\label{key}
\nabla_w\phi =\Q_h\nabla\phi.
\end{equation}
\end{lemma}
\begin{proof}
Using (\ref{d-d}) and  integration by parts, we have that for
any $\bq\in [P_{j}(T)]^d$
\begin{eqnarray*}
(\nabla_w \phi,\bq)_T &=& -(\phi,\nabla\cdot\bq)_T
+\langle \phi,\bq\cdot\bn\rangle_{\pT}\\
&=&(\nabla \phi,\bq)_T=(\Q_h\nabla\phi,\bq)_T,
\end{eqnarray*}
which implies the desired identity (\ref{key}).
\end{proof}

\section{Well Posedness}

For any $v\in V_h+H^1(\Omega)$, let
\begin{equation}\label{3barnorm}
\3bar v\3bar^2=(\nabla_wv,\nabla_wv)_{\T_h}.
\end{equation}

We introduce a discrete $H^1$ semi-norm as follows:
\begin{equation}\label{norm}
\|v\|_{1,h} = \left( \sum_{T\in\T_h}\left(\|\nabla
v_0\|_T^2+h_T^{-1} \|  v_0-v_b\|^2_\pT\right) \right)^{\frac12}.
\end{equation}
It is easy to see that $\|v\|_{1,h}$ define a norm in $V_h^0$. The following lemma indicates that $\|\cdot\|_{1,h}$ is equivalent
to the $\3bar\cdot\3bar$ in (\ref{3barnorm}).

\begin{lemma} There exist two positive constants $C_1$ and $C_2$ such
that for any $v=\{v_0,v_b\}\in V_h$, we have
\begin{equation}\label{happy}
C_1 \|v\|_{1,h}\le \3bar v\3bar \leq C_2 \|v\|_{1,h}.
\end{equation}
\end{lemma}

\medskip

\begin{proof}
For any $v=\{v_0,v_b\}\in V_h$, it follows from the definition of
weak gradient (\ref{d-d}) and integration by parts that
\begin{eqnarray}\label{n-1}
(\nabla_wv,\bq)_T=(\nabla v_0,\bq)_T+\l v_b-v_0,
\bq\cdot\bn\r_\pT,\quad \forall \bq\in [P_{j}(T)]^d.
\end{eqnarray}
By letting $\bq=\nabla_w v$ in (\ref{n-1}) we arrive at
\begin{eqnarray*}
(\nabla_wv,\nabla_w v)_T=(\nabla v_0,\nabla_w v)_T+\l v_b-v_0,
\nabla_w v\cdot\bn\r_\pT.
\end{eqnarray*}
From the trace inequality (\ref{trace}) and the inverse inequality
we have
\begin{eqnarray*}
\|\nabla_wv\|^2_T &\le& \|\nabla v_0\|_T \|\nabla_w v\|_T+ \|
v_0-v_b\|_\pT \|\nabla_w v\|_\pT\\
&\le& \|\nabla v_0\|_T \|\nabla_w v\|_T+ Ch_T^{-1/2}\|
v_0-v_b\|_\pT \|\nabla_w v\|_T,
\end{eqnarray*}
which implies
$$
\|\nabla_w v\|_T \le C \left(\|\nabla v_0\|_T +h_T^{-1/2}\|v_0-v_b\|_\pT\right),
$$
and consequently
$$\3bar v\3bar \leq C_2 \|v\|_{1,h}.$$

Next we will prove $C_1 \|v\|_{1,h}\le \3bar v\3bar $.
For $v\in V_h$ and $\bq\in [P_j(T)]^d$, by  \eqref{d-d} and integration by parts, we have
\begin{equation}\label{n2}
   (\nabla_w v,\bq)_T=(\nabla v_0,\bq)_T+\l  v_b-v_0, \bq\cdot\bn\r_\pT.
\end{equation}
Let $n$ be the number of the edges/faces on a polygon/polyhadron. It has been proved in \cite{cdg2} that there exists $\bq_0\in [P_j(T)]^d$, $j=n+k-1$, such that
\begin{equation}\label{2e}
(\nabla v_0,\bq_0)_T=0, \ \ \l v_b-v_0,\bq_0\cdot\bn\r_{\pT\setminus e}=0,
    \ \ \l v_b-v_0, \bq_0\cdot\bn\r_e=\|v_0-v_b\|_e^2,
\end{equation}
and
\begin{equation}
\|\bq_0\|_T \le C h_T^{1/2} \| v_b-v_0 \|_e.\label{22e}
\end{equation}
Substituting  $\bq_0$ into (\ref{n2}), we get
\begin{equation}\label{n3}
(\nabla_wv,\bq_0)_T=\|v_b-v_0\|^2_e.
\end{equation}
It follows from Cauchy-Schwarz inequality and (\ref{22e}) that
\[
\|v_b-v_0\|^2_e\le C\|\nabla_w v\|_T\|\bq_0\|_T
 \le Ch_T^{1/2}\|\nabla_w v\|_T\|v_0-v_b\|_e,
\]
which implies
\begin{equation}\label{n4}
h_T^{-1/2}\|v_0-v_b\|_\pT\le C\|\nabla_w v\|_T.
\end{equation}
It follows from the trace inequality, the inverse inequality and (\ref{n4}),
$$
\|\nabla v_0\|_T^2 \leq \|\nabla_w v\|_T \|\nabla v_0\|_T
+Ch_T^{-1/2}\| v_0-v_b\|_\pT \|\nabla v_0\|_T\le C\|\nabla_w v\|_T \|\nabla v_0\|_T.
$$
Combining the above estimate and (\ref{n4}),
   by the definition \eqref{norm},
 we prove  the lower bound of (\ref{happy}) and complete the proof of the lemma.
\end{proof}

\medskip

\begin{lemma}
The weak Galerkin finite element scheme (\ref{wg}) has a unique
solution.
\end{lemma}

\smallskip

\begin{proof}
If $u_h^{(1)}$ and $u_h^{(2)}$ are two solutions of (\ref{wg}), then
$\varepsilon_h=u_h^{(1)}-u_h^{(2)}\in V_h^0$ would satisfy the following equation
$$
(\nabla_w \varepsilon_h,\nabla_w v)=0,\qquad\forall v\in V_h^0.
$$
 Then by letting $v=\varepsilon_h$ in the above
equation we arrive at
$$
\3bar \varepsilon_h\3bar^2 = (\nabla_w \varepsilon_h,\nabla_w \varepsilon_h)=0.
$$
It follows from (\ref{happy}) that $\|\varepsilon_h\|_{1,h}=0$. Since $\|\cdot\|_{1,h}$ is a norm in $V_h^0$, one has $\varepsilon_h=0$.
 This completes the proof of the lemma.
\end{proof}

\section{Error Estimates in Energy Norm}

Let $e_h=u-u_h$ and $\epsilon_h=Q_hu-u_h$. Next we derive an error equation that $e_h$ satisfies.

\begin{lemma}
For any $v\in V_h^0$, the following error equation holds true
\begin{eqnarray}
(\nabla_we_h,\nabla_wv)_{\T_h}=\ell(u,v),\label{ee}
\end{eqnarray}
where
\begin{eqnarray*}
\ell(u,v)&=& \langle (\nabla u-\Q_h\nabla u)\cdot\bn,v_0-v_b\rangle_{\pT_h}.
\end{eqnarray*}
\end{lemma}

\begin{proof}
For $v=\{v_0,v_b\}\in V_h^0$, testing (\ref{pde}) by  $v_0$  and using the fact that
$\sum_{T\in\T_h}\langle \nabla u\cdot\bn, v_b\rangle_\pT=0$,  we arrive at
\begin{equation}\label{m1}
(\nabla u,\nabla v_0)_{\T_h}- \langle
\nabla u\cdot\bn,v_0-v_b\rangle_{\pT_h}=(f,v_0).
\end{equation}

It follows from integration by parts, (\ref{d-d}) and (\ref{key})  that
\begin{eqnarray}
(\nabla u,\nabla v_0)_{\T_h}&=&(\Q_h\nabla  u,\nabla v_0)_{\T_h}\nonumber\\
&=&-(v_0,\nabla\cdot (\Q_h\nabla u))_{\T_h}+\langle v_0, \Q_h\nabla u\cdot\bn\rangle_{\partial\T_h}\nonumber\\
&=&(\Q_h\nabla u, \nabla_w v)_{\T_h}+\langle v_0-v_b,\Q_h\nabla u\cdot\bn\rangle_{\partial\T_h}\nonumber\\
&=&( \nabla_w u, \nabla_w v)_{\T_h}+\langle v_0-v_b,\Q_h\nabla u\cdot\bn\rangle_{\partial\T_h}.\label{j1}
\end{eqnarray}
Combining (\ref{m1}) and (\ref{j1}) gives
\begin{eqnarray}
(\nabla_w u,\nabla_w v)_{\T_h}&=&(f,v_0)+\ell(u,v).\label{j2}
\end{eqnarray}
The error equation follows from subtracting (\ref{wg}) from (\ref{j2}),
\begin{eqnarray*}
(\nabla_we_h,\nabla_wv)_{\T_h}=\ell(u,v),\quad \forall v\in V_h^0.
\end{eqnarray*}
This completes the proof of the lemma.
\end{proof}

For any function $\varphi\in H^1(T)$, the following trace
inequality holds true (see \cite{wymix} for details):
\begin{equation}\label{trace}
\|\varphi\|_{e}^2 \leq C \left( h_T^{-1} \|\varphi\|_T^2 + h_T
\|\nabla \varphi\|_{T}^2\right).
\end{equation}

\begin{lemma} For any $w\in H^{k+1}(\Omega)$ and
$v=\{v_0,v_b\}\in V_h^0$, we have
\begin{eqnarray}
|\ell(w, v)|&\le&
Ch^{k}|w|_{k+1}\3bar v\3bar.\label{mmm1}
\end{eqnarray}
\end{lemma}

\medskip

\begin{proof}
Using the Cauchy-Schwarz inequality, the trace inequality (\ref{trace}) and (\ref{happy}), we have
\begin{eqnarray*}
|\ell(w,v)|&=&\left|\sum_{T\in\T_h}\langle (\nabla w-\Q_h\nabla
w)\cdot\bn, v_0-v_b\rangle_\pT\right|\\
&\le & C \sum_{T\in\T_h}\|(\nabla w-\Q_h\nabla w)\|_{\pT}
\|v_0-v_b\|_\pT\nonumber\\
&\le & C \left(\sum_{T\in\T_h}h_T\|(\nabla w-\Q_h\nabla w)\|_{\pT}^2\right)^{\frac12}
\left(\sum_{T\in\T_h}h_T^{-1}\|v_0-v_b\|_\pT^2\right)^{\frac12}\\
&\le & Ch^{k}|w|_{k+1}\3bar v\3bar,
\end{eqnarray*}
which proves the lemma.
\end{proof}

\smallskip

\begin{lemma}
Let $w\in H^{k+1}(\Omega)$, then
\begin{equation}\label{eee2}
\3bar w-Q_hw\3bar\le Ch^k|w|_{k+1}.
\end{equation}
\end{lemma}
\begin{proof}
It follows from (\ref{d-d}), integration  by parts, and (\ref{trace}),
\begin{eqnarray*}
(\nabla_w(w-Q_hw), \bq)_{T}&=&-(w-Q_0w, \nabla\cdot\bq)_{T}+\l w-Q_bw, \bq\cdot\bn\r_{\pT}\\
&=&(\nabla (w-Q_0w), \bq)_{T}+\l Q_0w-Q_bw, \bq\cdot\bn\r_{\pT}\\
&\le& \|\nabla (w-Q_0w)\|_T\|\bq\|_T+Ch^{-1/2}\|w-Q_0w\|_\pT\|\bq\|_T\\
&\le& Ch^k|w|_{k+1, T}\|\bq\|_T.
\end{eqnarray*}
Letting $\bq=\nabla_w(w-Q_hw)$ in the above equation and taking summation over $T$, we have
\[
\3bar w-Q_hw\3bar\le Ch^k|w|_{k+1}.
\]
We have proved the lemma.
\end{proof}

\begin{theorem} Let $u_h\in V_h$ be the weak Galerkin finite element solution of (\ref{wg}). Assume the exact solution $u\in H^{k+1}(\Omega)$. Then,
there exists a constant $C$ such that
\begin{equation}\label{err1}
\3bar u-u_h\3bar \le Ch^{k}|u|_{k+1}.
\end{equation}
\end{theorem}
\begin{proof}
It is straightforward to obtain
\begin{eqnarray}
\3bar e_h\3bar^2&=&(\nabla_we_h, \nabla_we_h)_{\T_h}\label{eee1}\\
&=&(\nabla_wu-\nabla_wu_h,\nabla_we_h)_{\T_h}\nonumber\\
&=&(\nabla_wQ_hu-\nabla_wu_h,\nabla_we_h)_{\T_h}+(\nabla_wu-\nabla_wQ_hu,\nabla_we_h)_{\T_h}\nonumber\\
&=&(\nabla_we_h,\nabla_w\epsilon_h)_{\T_h}+(\nabla_wu-\nabla_wQ_hu,\nabla_we_h)_{\T_h}.\nonumber
\end{eqnarray}
We will bound each terms in (\ref{eee1}).
Letting $v=\epsilon_h\in V_h^0$ in (\ref{ee})  and using (\ref{mmm1}) and (\ref{eee2}), we have
\begin{eqnarray}
|(\nabla_we_h,\nabla_w\epsilon_h)_{\T_h}|&=&|\ell(u,\epsilon_h)|\nonumber\\
&\le& Ch^{k}|u|_{k+1}\3bar \epsilon_h\3bar\nonumber\\
&\le& Ch^{k}|u|_{k+1}\3bar Q_hu-u_h\3bar\nonumber\\
&\le& Ch^{k}|u|_{k+1}(\3bar Q_hu-u\3bar+\3bar u-u_h\3bar)\nonumber\\
&\le& Ch^{2k}|u|^2_{k+1}+\frac14 \3bare_h\3bar^2.\label{eee3}
\end{eqnarray}
The estimate (\ref{eee2}) implies
\begin{eqnarray}
|(\nabla_wu-\nabla_wQ_hu,\nabla_we_h)_{\T_h}|&\le& C\3bar u-Q_hu\3bar \3bar e_h\3bar\nonumber\\
&\le& Ch^{2k}|u|^2_{k+1}+\frac14\3bar e_h\3bar^2.\label{eee4}
\end{eqnarray}
Combining the estimates (\ref{eee3}) and  (\ref{eee4}) with (\ref{eee1}), we arrive
\[
\3bar e_h\3bar \le Ch^{k}|u|_{k+1},
\]
which completes the proof.
\end{proof}

\section{Error Estimates in $L^2$ Norm}

The standard duality argument is used to obtain $L^2$ error estimate.
Recall $e_h=\{e_0,e_b\}=u-u_h$ and $\epsilon_h=\{\epsilon_0,\epsilon_b\}=Q_hu-u_h$.
The considered dual problem seeks $\Phi\in H_0^1(\Omega)$ satisfying
\begin{eqnarray}
-\Delta\Phi&=& \epsilon_0,\quad
\mbox{in}\;\Omega.\label{dual}
\end{eqnarray}
Assume that the following $H^{2}$-regularity holds
\begin{equation}\label{reg}
\|\Phi\|_2\le C\|\epsilon_0\|.
\end{equation}

\begin{theorem} Let $u_h\in V_h$ be the weak Galerkin finite element solution of (\ref{wg}). Assume that the
exact solution $u\in H^{k+1}(\Omega)$ and (\ref{reg}) holds true.
 Then, there exists a constant $C$ such that
\begin{equation}\label{err2}
\|u-u_0\| \le Ch^{k+1}|u|_{k+1}.
\end{equation}
\end{theorem}

\begin{proof}
Testing (\ref{dual}) by $e_0$ and using the fact that $\sum_{T\in\T_h}\langle \nabla
\Phi\cdot\bn, \epsilon_b\rangle_\pT=0$ give
\begin{eqnarray}\nonumber
\|\epsilon_0\|^2&=&-(\Delta\Phi,\epsilon_0)\\
&=&(\nabla \Phi,\ \nabla \epsilon_0)_{\T_h}-\l
\nabla\Phi\cdot\bn,\ \epsilon_0- \epsilon_b\r_{\pT_h}.\label{jw.08}
\end{eqnarray}
Setting $u=\Phi$ and $v=\epsilon_h$ in (\ref{j1}) yields
\begin{eqnarray}
(\nabla\Phi,\;\nabla \epsilon_0)_{\T_h}=(\nabla_w \Phi,\;\nabla_w \epsilon_h)_{\T_h}+\l
\Q_h\nabla\Phi\cdot\bn,\ \epsilon_0-\epsilon_b\r_{\pT_h}.\label{j1-new}
\end{eqnarray}
Substituting (\ref{j1-new}) into (\ref{jw.08}) gives
\begin{eqnarray}
\|\epsilon_0\|^2&=&(\nabla_w \epsilon_h,\ \nabla_w\Phi)_{\T_h}-\l
(\nabla\Phi-\Q_h\nabla\Phi)\cdot\bn,\ \epsilon_0-\epsilon_b\r_{\pT_h}\nonumber\\
&=&(\nabla_w e_h,\ \nabla_w\Phi)_{\T_h}+(\nabla_w (Q_hu-u),\ \nabla_w\Phi)_{\T_h}+\ell(\Phi, \epsilon_h)\nonumber\\
&=&(\nabla_w e_h,\ \nabla_wQ_h\Phi)_{\T_h}+(\nabla_w e_h,\ \nabla_w(\Phi-Q_h\Phi))_{\T_h}\nonumber\\
&+&(\nabla_w (Q_hu-u),\ \nabla_w\Phi)_{\T_h}+\ell(\Phi, \epsilon_h)\nonumber\\
&=&\ell(u,Q_h\Phi)+(\nabla_w e_h,\ \nabla_w(\Phi-Q_h\Phi))_{\T_h}+(\nabla_w (Q_hu-u),\ \nabla_w\Phi)_{\T_h}+\ell(\Phi, \epsilon_h)\nonumber\\
&=&I_1+I_2+I_3+I_4.\label{m2}
\end{eqnarray}

Next we will estimate all the terms on the right hand side of (\ref{m2}). Using the Cauchy-Schwarz inequality, the trace inequality (\ref{trace}) and the definitions of $Q_h$ and $\Pi_h$
we obtain
\begin{eqnarray*}
I_1&=&|\ell(u,Q_h\Phi)|\le\left| \langle (\nabla u-\Q_h\nabla
u)\cdot\bn,\;
Q_0\Phi-Q_b\Phi\rangle_{\pT_h} \right|\\
&\le& \left(\sum_{T\in\T_h}\|(\nabla u-\Q_h\nabla
u)\|^2_\pT\right)^{1/2}
\left(\sum_{T\in\T_h}\|Q_0\Phi-Q_b\Phi\|^2_\pT\right)^{1/2}\nonumber \\
&\le& C\left(\sum_{T\in\T_h}h\|(\nabla u-\Q_h\nabla
u)\|^2_\pT\right)^{1/2}
\left(\sum_{T\in\T_h}h^{-1}\|Q_0\Phi-\Phi\|^2_\pT\right)^{1/2} \nonumber\\
&\le&  Ch^{k+1}|u|_{k+1}|\Phi|_2.\nonumber
\end{eqnarray*}
It follows from (\ref{err1}) and (\ref{eee2}) that
\begin{eqnarray*}
I_2&=&|(\nabla_w e_h,\ \nabla_w(\Phi-Q_h\Phi))_{\T_h}|\le C\3bar e_h\3bar \3bar \Phi-Q_h\Phi\3bar\\
&\le& Ch^{k+1}|u|_{k+1}|\Phi|_2.
\end{eqnarray*}
To bound $I3$, we define a $L^2$ projection element-wise onto $[P_1(T)]^d$ denoted by $R_h$. Then it follows from the definition of weak gradient (\ref{d-d})
\begin{eqnarray*}
(\nabla_w (Q_hu-u),\ R_h\nabla_w\Phi)_{T}&=&-(Q_0u-u,\nabla\cdot R_h\nabla_w\Phi)_T+ \l (Q_bu-u, R_h\nabla_w\Phi\cdot\bn\r_\pT=0
\end{eqnarray*}
Using the equation above and (\ref{eee2}) and the definition of $R_h$, we have
\begin{eqnarray*}
I_3&=&|(\nabla_w (Q_hu-u),\ \nabla_w\Phi)_{\T_h}|\\
&=&|(\nabla_w (Q_hu-u),\ \nabla_w\Phi-R_h\nabla_w\Phi)_{\T_h}|\\
&=&|(\nabla_w (Q_hu-u),\ \nabla\Phi-R_h\nabla\Phi)_{\T_h}|\\
&\le& Ch^{k+1}|u|_{k+1}|\Phi|_2.
\end{eqnarray*}
It follows from (\ref{mmm1}), (\ref{eee2}) and (\ref{err1}) that
\begin{eqnarray*}
I_4&=&|\ell(\Phi,\epsilon_h)|\le Ch|\Phi|_2\3bar \epsilon_h\3bar\\
&\le& Ch|\Phi|_2(\3bar e_h\3bar+\3bar u-Q_hu\3bar)\\
&\le&  Ch^{k+1 }|u|_{k+1}\|\Phi\|_2.
\end{eqnarray*}
Combining all the estimates above
with (\ref{m2}) yields
$$
\|\epsilon_0\|^2 \leq C h^{k+1}|u|_{k+1} \|\Phi\|_2.
$$
It follows from the above inequality and
the regularity assumption (\ref{reg}).
 $$
\|\epsilon_0\|\leq C h^{k+1}|u|_{k+1}.
$$
The triangle inequality implies
$$
\|e_0\|\le \|\epsilon_0\|+\|u-Q_0u\| \leq C h^{k+1}|u|_{k+1}.
$$
We have completed the proof.
\end{proof}

\section{Numerical Experiments}\label{Section:numerical-experiments}

 We solve the following Poisson equation on the unit square:
\begin{align} \label{s1} -\Delta u = 2\pi^2 \sin\pi x\sin \pi y,  \quad (x,y)\in\Omega=(0,1)^2,
\end{align} with the boundary condition $u=0$ on $\partial \Omega$.

\begin{figure}[h!]
 \begin{center} \setlength\unitlength{1.25pt}
\begin{picture}(260,80)(0,0)
  \def\tr{\begin{picture}(20,20)(0,0)\put(0,0){\line(1,0){20}}\put(0,20){\line(1,0){20}}
          \put(0,0){\line(0,1){20}} \put(20,0){\line(0,1){20}}  \put(20,0){\line(-1,1){20}}\end{picture}}
 {\setlength\unitlength{5pt}
 \multiput(0,0)(20,0){1}{\multiput(0,0)(0,20){1}{\tr}}}

  {\setlength\unitlength{2.5pt}
 \multiput(45,0)(20,0){2}{\multiput(0,0)(0,20){2}{\tr}}}

  \multiput(180,0)(20,0){4}{\multiput(0,0)(0,20){4}{\tr}}

 \end{picture}\end{center}
\caption{\label{grid1} The first three levels of grids used in the computation of Table \ref{t1}. }
\end{figure}
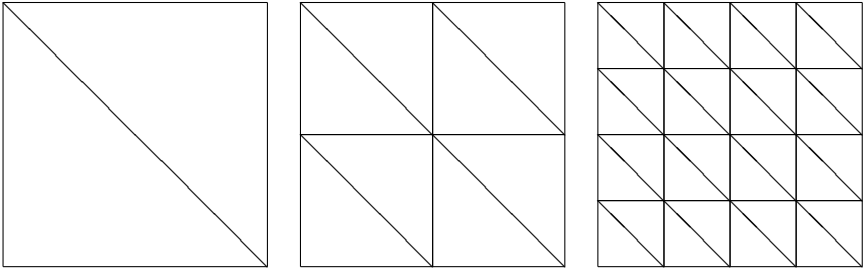

In the first computation, the level one grid consists of two unit right triangles
     cutting from the unit square by a forward
  slash.   The high level grids are the half-size refinements of the previous grid.
The first three levels of grids are plotted in Figure \ref{grid1}.
The error and the order of convergence are shown in Table \ref{t1}.
The numerical results confirm the convergence theory.

\begin{table}[h!]
  \centering \renewcommand{\arraystretch}{1.1}
  \caption{Error profiles and convergence rates for \eqref{s1} on triangular grids }\label{t1}
\begin{tabular}{c|cc|cc}
\hline
level & $\|u_h- Q_0 u\| $  &rate & $\3bar u_h- u\3bar $ &rate \\
\hline
 &\multicolumn{4}{l}{by $P_1$ elements with $P_1^2$ weak gradient $\Rightarrow$ singular} \\ \hline
 &\multicolumn{4}{l}{by $P_1$ elements with $P_2^2$ weak gradient} \\ \hline
 6&   0.4295E-03 & 1.99&   0.5369E-01 & 1.00\\
 7&   0.1075E-03 & 2.00&   0.2684E-01 & 1.00\\
 8&   0.2688E-04 & 2.00&   0.1342E-01 & 1.00\\
 \hline
 &\multicolumn{4}{l}{by $P_2$ elements with $P_2^2$ weak gradient $\Rightarrow$ singular} \\ \hline
 &\multicolumn{4}{l}{by $P_2$ elements with $P_3^2$ weak gradient} \\ \hline
 6&   0.2383E-05 & 3.01&   0.1013E-02 & 2.00\\
 7&   0.2971E-06 & 3.00&   0.2532E-03 & 2.00\\
 8&   0.3709E-07 & 3.00&   0.6330E-04 & 2.00\\
 \hline
 &\multicolumn{4}{l}{by $P_3$ elements with $P_3^2$ weak gradient $\Rightarrow$ singular} \\ \hline
 &\multicolumn{4}{l}{by $P_3$ elements with $P_4^2$ weak gradient} \\ \hline
 6&   0.2468E-07 & 4.02&   0.1430E-04 & 3.00\\
 7&   0.1532E-08 & 4.01&   0.1789E-05 & 3.00\\
 8&   0.9550E-10 & 4.00&   0.2237E-06 & 3.00\\
 \hline
 &\multicolumn{4}{l}{by $P_4$ elements with $P_4^2$ weak gradient $\Rightarrow$ singular} \\ \hline
 &\multicolumn{4}{l}{by $P_4$ elements with $P_5^2$ weak gradient} \\ \hline
 5&   0.8154E-08 & 4.99&   0.2441E-05 & 4.00\\
 6&   0.2551E-09 & 5.00&   0.1526E-06 & 4.00\\
 7&   0.8257E-11 & 4.99&   0.9539E-08 & 4.00\\
 \hline
\end{tabular}%
\end{table}%

In the next computation,  we use a family of polygonal grids (with 12-side polygons)
   shown in Figure \ref{12gon}.
The numerical results in  Table \ref{t3} indicate  that the polynomial degree $j$ for the weak gradient needs to
   be larger, which confirms the theory: $j$ depending on the number of edges of a polygon.
The convergence history confirms the theory.

\begin{figure}[htb]\begin{center}\setlength\unitlength{1.5in}
    \begin{picture}(3.2,1.4)
 \put(0,0){\includegraphics[width=1.5in]{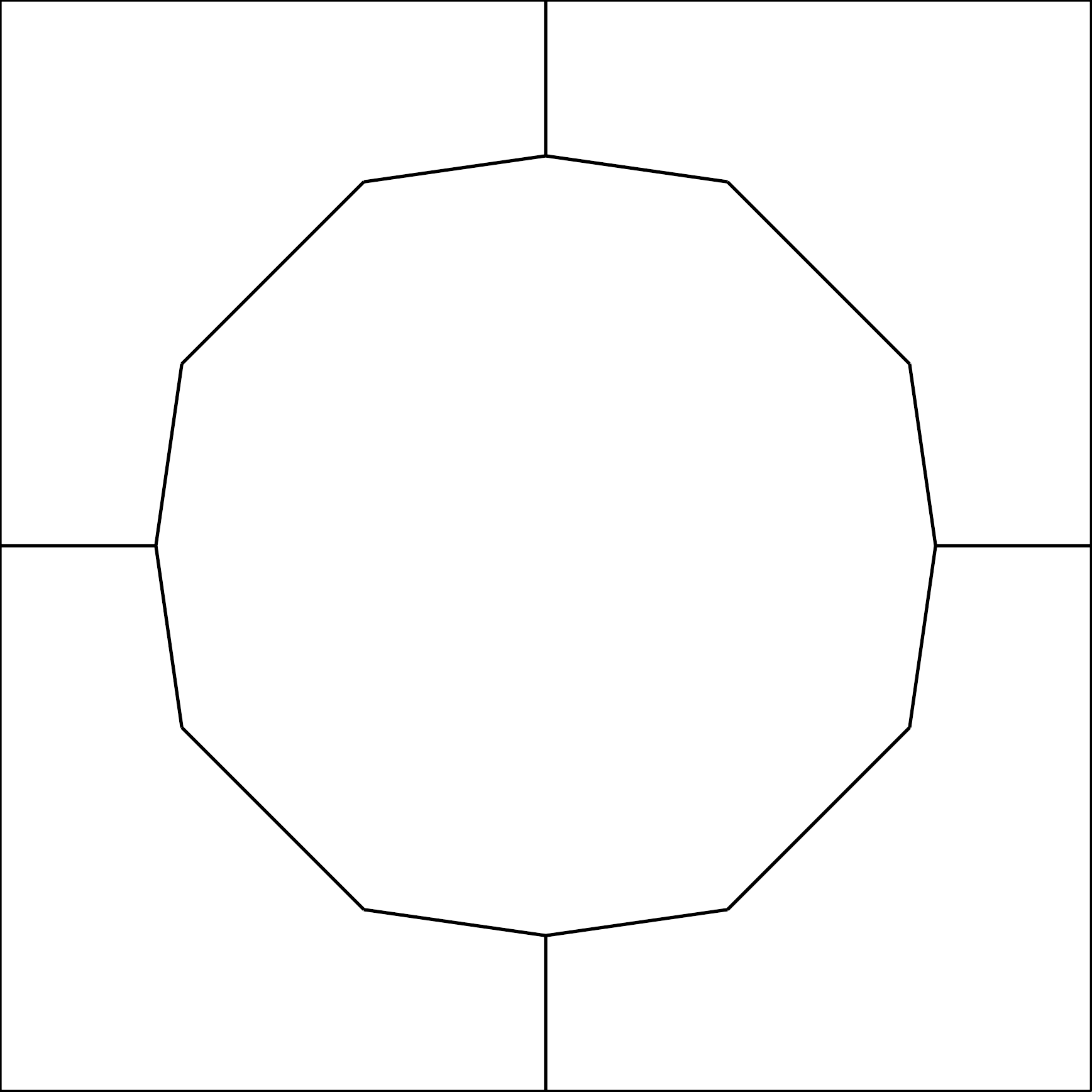}}  
 \put(1.1,0){\includegraphics[width=1.5in]{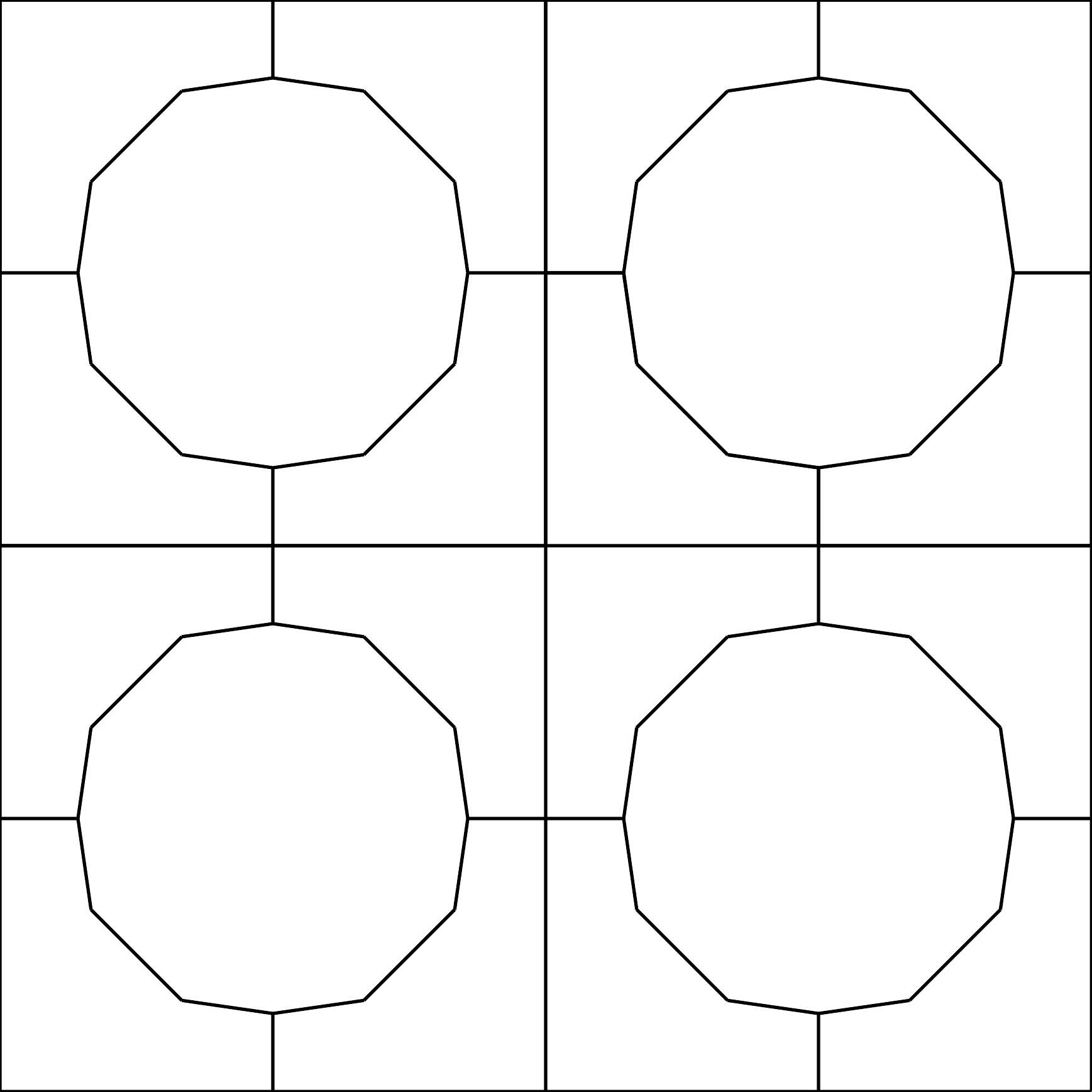}}
 \put(2.2,0){\includegraphics[width=1.5in]{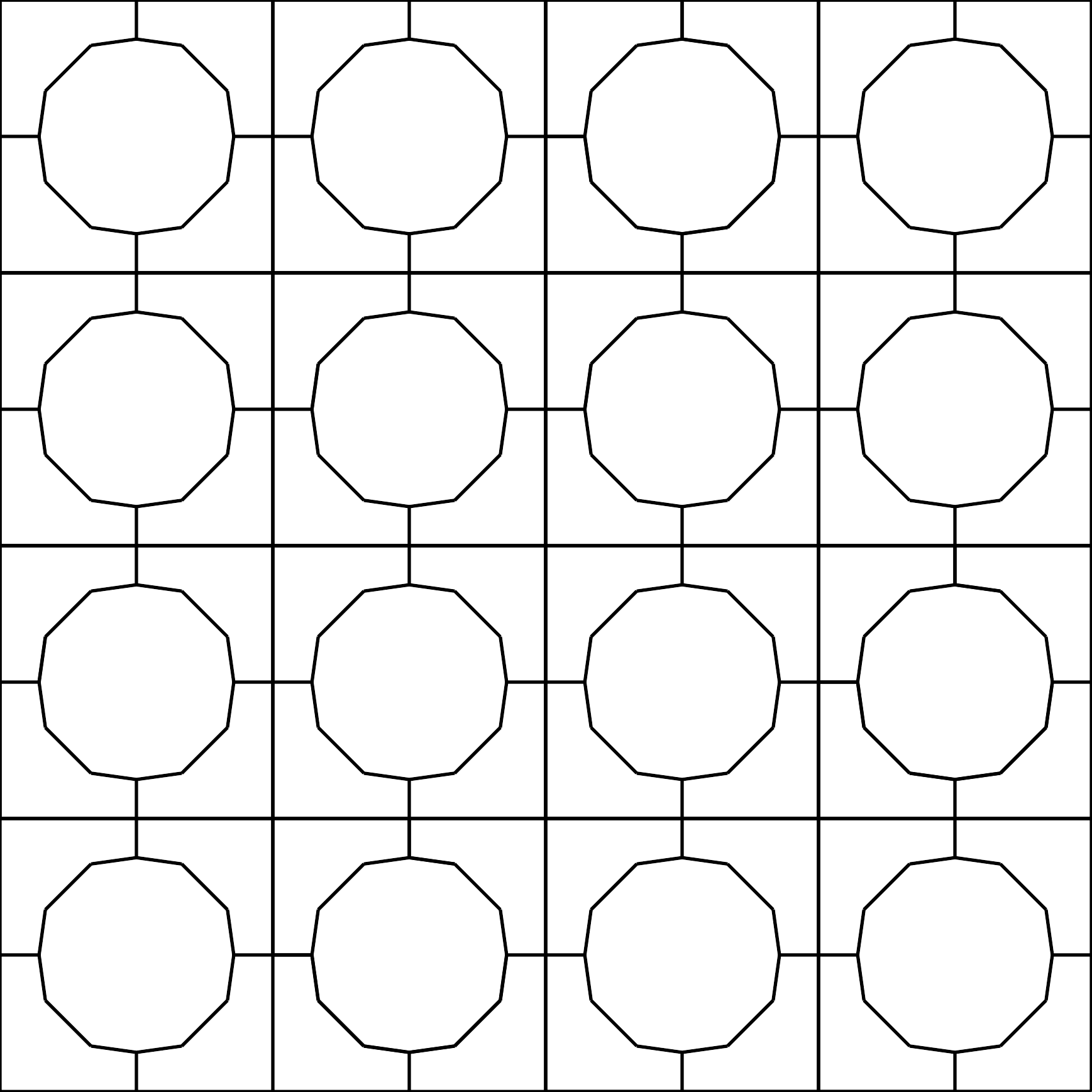}}
    \end{picture}
\caption{ The first three polygonal grids for the computation of Table \ref{t3}.  } \label{12gon}
\end{center}
\end{figure}

\begin{table}[h!]
  \centering \renewcommand{\arraystretch}{1.1}
  \caption{Error profiles and convergence rates for \eqref{s1} on polygonal grids
    shown in Figure \ref{12gon} }\label{t3}
\begin{tabular}{c|cc|cc}
\hline
level & $\|u_h- Q_0u\| $  &rate & $\3bar u_h-u\3bar $ &rate \\
  \hline
  &\multicolumn{4}{l}{by $P_1$ elements with $P_2^2$ weak gradient  $\Rightarrow$ singular} \\ \hline
  &\multicolumn{4}{l}{by $P_1$ elements with $P_3^2$ weak gradient} \\ \hline
 5&   0.9671E-03 & 1.98&   0.1350E+00 & 1.00 \\
 6&   0.2425E-03 & 2.00&   0.6750E-01 & 1.00 \\
 7&   0.6067E-04 & 2.00&   0.3375E-01 & 1.00 \\

  \hline
  &\multicolumn{4}{l}{by $P_2$ elements with $P_3^2$ weak gradient  $\Rightarrow$ singular} \\ \hline
  &\multicolumn{4}{l}{by $P_2$ elements with $P_4^2$ weak gradient} \\ \hline
 5&   0.5791E-05 & 3.00&   0.3247E-02 & 2.00 \\
 6&   0.7233E-06 & 3.00&   0.8120E-03 & 2.00 \\
 7&   0.9040E-07 & 3.00&   0.2030E-03 & 2.00 \\
 \hline
  &\multicolumn{4}{l}{by $P_3$ elements with $P_4^2$ weak gradient  $\Rightarrow$ singular} \\ \hline
  &\multicolumn{4}{l}{by $P_3$ elements with $P_5^2$ weak gradient} \\ \hline
 4&   0.8809E-06 & 4.00&   0.3575E-03 & 2.99 \\
 5&   0.5509E-07 & 4.00&   0.4475E-04 & 3.00 \\
 6&   0.3447E-08 & 4.00&   0.5595E-05 & 3.00 \\
\hline
    \end{tabular}%
\end{table}%

\end{document}